\documentclass[fleqn,12pt,twoside]{article}

\usepackage[headings]{espcrc1}
\readRCS
$Id: espcrc1.tex,v 1.2 2004/02/24 11:22:11 spepping Exp $
\usepackage{graphicx}
\usepackage[figuresright]{rotating}

\usepackage{amsmath}
\usepackage{amsfonts}
\usepackage{amssymb}
\usepackage{amsthm}

\usepackage{tikz}
\usetikzlibrary{arrows}
\usetikzlibrary{shapes}
\usetikzlibrary{decorations.markings}

\usetikzlibrary{positioning,chains,fit,shapes,calc}
\usepackage{hyperref}

\newtheorem{theorem}{Theorem}

\newtheorem{corollary}[theorem]{Corollary}

\newtheorem{lemma}[theorem]{Lemma}

\numberwithin{theorem}{section}

\title{On the deficiency of complete multipartite graphs}
\author{A.R. Davtyan\address[MCSD]{Department of Informatics and Applied Mathematics,\\
Yerevan State University, 0025, Armenia}%
\thanks{email: armen2davtyan@gmail.com},
        G.M. Minasyan\addressmark[MCSD]%
\thanks{email: gevor.minasyan94@gmail.com},
        P.A. Petrosyan\addressmark[MCSD]%
\address{Department of Applied Mathematics and Informatics,\\
Russian-Armenian University, 0051, Armenia}%
\thanks{email: petros\_petrosyan@ysu.am}}

\runtitle{On the deficiency of complete multipartite graphs}\runauthor{A.R. Davtyan, G.M. Minasyan, P.A. Petrosyan}

\begin{document}

\maketitle

\begin{abstract}
An edge-coloring of a graph $G$ with colors $1,\ldots,t$ is an \emph{interval $t$-coloring} if all colors are used, and the colors of edges incident to each vertex of $G$ are distinct and form an integer interval. It is well-known that there are graphs that do not have interval colorings. The \emph{deficiency} of a graph $G$, denoted by $\mathrm{def}(G)$, is the minimum number of pendant edges whose attachment to $G$ leads to a graph admitting an interval coloring. In this paper we investigate the problem of determining or bounding of the deficiency of complete multipartite graphs. In particular, we obtain a tight upper bound for the deficiency of complete multipartite graphs. We also determine or bound the deficiency for some classes of complete multipartite graphs.\\

\end{abstract}

\section{Introduction}\

All graphs considered in this paper are finite, undirected, and have no loops or multiple edges. Let $V(G)$ and $E(G)$ denote the sets of vertices and edges of $G$, respectively. If $S\subseteq V(G)$, then $G[S]$ denotes the subgraph of $G$ induced by $S$. The degree of a vertex $v\in V(G)$ is denoted by $d_{G}(v)$, the maximum degree of vertices in $G$ by $\Delta(G)$, and the chromatic index of $G$ by $\chi^{\prime}(G)$. The terms and concepts that we do not define can be found in \cite{AsrDenHag,Kubale,West}.

A \emph{proper edge-coloring} of a graph $G$ is a mapping
$\alpha: E(G)\rightarrow \mathbb{N}$ such that  $\alpha(e)\neq \alpha(e^{\prime})$ for every pair of adjacent edges $e,e^{\prime}\in E(G)$. If $\alpha $ is a proper edge-coloring of a graph $G$ and $v\in V(G)$, then the \emph{spectrum of a vertex $v$}, denoted by $S_{G}\left(v,\alpha \right)$ (or $S\left(v,\alpha \right)$), is the set of all colors appearing on edges incident to $v$. A proper edge-coloring $\alpha$ of a graph $G$ with colors $1,\ldots,t$ is an \emph{interval $t$-coloring} if for each vertex $v$ of $G$, the set $S\left(v,\alpha \right)$ is an interval of integers. A graph $G$ is \emph{interval colorable} if it has an interval $t$-coloring for some positive integer $t$. The set of all interval colorable graphs is denoted by $\mathfrak{N}$. The concept of interval edge-coloring of graphs was introduced by Asratian and Kamalian \cite{AsrKam} in 1987. In \cite{AsrKam}, the authors noted that if $G\in \mathfrak{N}$, then $\chi^{\prime}\left(G\right)=\Delta(G)$.
Asratian and Kamalian also proved \cite{AsrKam,AsrKamJCTB} that if a triangle-free graph $G$ admits an interval $t$-coloring, then $t\leq \left\vert V(G)\right\vert -1$. Generally, it is an $NP$-complete problem to determine whether a bipartite graph has an interval coloring \cite{Seva}. In \cite{Kampreprint,KamDiss},
Kamalian investigated interval colorings of complete bipartite
graphs and trees. In particular, he proved that the complete
bipartite graph $K_{m,n}$ has an interval $t$-coloring if and only
if $m+n-\gcd(m,n)\leq t\leq m+n-1$, where $\gcd(m,n)$ is the
greatest common divisor of $m$ and $n$. In \cite{PetDM,PetKhachTan},
Petrosyan, Khachatrian and Tananyan proved that the $n$-dimesional
cube $Q_{n}$ has an interval $t$-coloring if and only if $n\leq t\leq \frac{n\left(n+1\right)}{2}$. There are many papers devoted to this topic, in particular, surveys on the topic can be found in some books \cite{AsrDenHag,Kubale}.\

There are graphs that have no interval colorings; a smallest example is $K_{3}$. Since not all graphs admit an interval coloring, it is naturally to consider a measure of closeness for a graph to be interval colorable. In
\cite{GiaroKubaleMalaf1}, Giaro, Kubale and Ma\l afiejski introduced such a measure which is called deficiency of a graph. The \emph{deficiency $\mathrm{def}(G)$ of a graph $G$} is the minimum number of pendant edges whose attachment to $G$ makes it interval colorable.
The concept of deficiency of graphs can be also defined in terms of proper edge-colorings. The \emph{deficiency of a proper edge-coloring $\alpha$ at vertex $v\in V(G)$}, denoted by $\mathrm{def}(v,\alpha)$, is the minimum number of integers which must be added to $S\left(v,\alpha \right)$ to form an interval, and the \emph{deficiency $\mathrm{def}\left(G,\alpha\right)$ of a proper edge-coloring $\alpha$ of $G$} is defined as the sum $\sum_{v\in V(G)}\mathrm{def}(v,\alpha)$. So, $\mathrm{def}(G)=\min_{\alpha}\mathrm{def}\left(G,\alpha\right)$, where minimum is taken over all possible proper edge-colorings of $G$. Clearly,
$\mathrm{def}(G)=0$ if and only if $G\in\mathfrak{N}$. 
In general, the problem of determining the deficiency of a graph is
$NP$-complete, even for regular and bipartite graphs
\cite{AsrKam,Seva,GiaroKubaleMalaf1}. In \cite{GiaroKubaleMalaf1},
Giaro, Kubale and Ma\l afiejski obtained some results on the
deficiency of bipartite graphs. In particular, they showed that
there are bipartite graphs whose deficiency approaches the number of vertices. In \cite{GiaroKubaleMalaf2}, the same authors proved that if $G$ is an $r$-regular graph with an odd number of vertices, then $\mathrm{def}(G)\geq \frac{r}{2}$, and determined the deficiency of odd cycles, complete graphs, wheels and broken wheels. In \cite{Schwartz}, Schwartz studied the deficiency of regular graphs. In particular, he obtained tight bounds on the deficiency of regular graphs and proved that there are regular graphs with high deficiency. Bouchard, Hertz and Desaulniers \cite{BouchHertzDesau} derived some lower bounds on the deficiency of graphs and provided a tabu search algorithm for finding a proper edge-coloring with minimum deficiency of a graph. Borowiecka-Olszewska, Drgas-Burchardt and Ha\l uszczak \cite{B-OD-BHal} studied the deficiency of $k$-trees. In particular, they determined the deficiency of all $k$-trees with maximum degree at most $2k$, where
$k\in \{2,3,4\}$. They also proved that the following lower bound
on $\mathrm{def}(G)$ holds: if $G$ is a graph with an odd number of
vertices, then $\mathrm{def}(G)\geq \frac{2\vert E(G)\vert -(\vert V(G)\vert
-1)\Delta(G)}{2}$. Recently, Khachatrian \cite{KhachOuterplanar} proved that the following upper bound on $\mathrm{def}(G)$ for outerplanar graphs holds: if $G$ is an outerplanar graph, then $\mathrm{def}(G)\leq \frac{|V(G)|-2}{og(G)-2}$, where $og(G)$ is the odd girth of the graph.    

One of the less-investigated problems related to the deficiency of graphs is the problem of determining the deficiency of complete multipartite graphs. Several special cases of complete multipartite graphs were considered by some authors. In partitcular, Giaro, Kubale and Ma\l afiejski \cite{GiaroKubaleMalaf2} proved the following result for the complete graph $K_{n}$: $\mathrm{def}(K_{n})=0$ if $n$ is even, and $\mathrm{def}(K_{n})=\frac{n-1}{2}$ if $n$ is odd. Later, Petrosyan and Khachatrian \cite{PetrosHrant} proved that for near-complete graphs $\mathrm{def}(K_{2n+1}-e)=n-1$ (where $e$ is an edge of $K_{2n+1}$), thereby confirming a conjecture of Borowiecka-Olszewska, Drgas-Burchardt and Ha\l uszczak \cite{B-OD-BHal}. They also proved the following result for the complete tripartite graph $K_{1,m,n}$: $\mathrm{def}(K_{1,m,n})=0$ if $\gcd(m+1,n+1)=1$, and $\mathrm{def}(K_{1,m,n})=1$ otherwise. 

In this paper, we obtain a tight upper bound for the deficiency of complete multipartite graphs. We also determine or bound the deficiency for some classes of complete multipartite graphs.\\

\section{Notation, definitions and auxiliary results}\

A graph $G$ is called a complete $r$-partite ($r\geq 2$) graph if
its vertices can be partitioned into $r$ nonempty independent sets
$V_1,\ldots,V_r$ such that each vertex in $V_i$ is adjacent to all
the other vertices in $V_j$ for $1\leq i<j\leq r$. Let
$K_{n_{1},n_{2},\ldots,n_{r}}$ denote a complete $r$-partite graph
with independent sets $V_1,V_2,\ldots,V_r$ of sizes
$n_{1},n_{2},\ldots,n_{r}$.

Let $G$ and $H$ be two graphs. The join $G\vee H$ of graphs $G$ and $H$ is defined as follows:
\begin{center}
$V(G\vee H)=V(G)\cup V(H)$,\\
$E(G\vee H)=E(G)\cup E(H)\cup\left\{uv\colon\,u\in
V(G)\wedge v\in V(H)\right\}$.
\end{center}

For two positive integers $a$ and $b$ with $a\leq b$, we denote by
$\left[a,b\right]$ the interval of integers
$\left\{a,a+1,\ldots,b-1,b\right\}$. If $a>b$, then $\left[a,b\right]=\emptyset$.
For an interval $[a,b]$ and a nonnegative number $p$, the notation $[a,b]\oplus p$ means: $[a + p, b + p]$.

Let $L=\left(l_{1},\ldots,l_{k}\right)$ be an ordered sequence of
nonnegative integers. The smallest and largest elements of $L$ are
denoted by $\underline L$ and $\overline L$, respectively. The
length (the number of elements) of $L$ is denoted by $\vert L\vert$. An
ordered sequence $L=\left(l_{1},\ldots,l_{k}\right)$ is called a
\emph{continuous sequence} if it contains all integers between
$\underline L$ and $\overline L$.

Let $A$ be a finite set of integers. The deficiency $\mathrm{def}(A)$ of $A$
is the number of integers between $\min A$ and $\max A$ not
belonging to $A$. Clearly, $\mathrm{def}(A)=\max A-\min A-\vert A\vert +1$. A set $A$ with $\mathrm{def}(A)=0$ is an integer interval. Note that if $\alpha$ is a proper edge-coloring of $G$ and $v \in V(G)$, then $\mathrm{def}(v,\alpha)=\mathrm{def}\left(S\left(v,\alpha \right)\right)$.

If $\alpha $ is a proper edge-coloring of a graph $G$ and $v\in
V(G)$, then the smallest and largest colors
of the spectrum $S\left(v,\alpha \right)$ are denoted by $\underline
S\left(v,\alpha \right)$ and $\overline S\left(v,\alpha \right)$,
respectively. Let $\alpha$ be a proper edge-coloring of $G$ and
$V^{\prime}=\{v_{1},\ldots,v_{k}\}\subseteq V(G)$. Consider the sets
$S\left(v_{1},\alpha \right),\ldots,S\left(v_{k},\alpha \right)$.
For a coloring $\alpha$ of $G$ and $V^{\prime}\subseteq V(G)$,
define two ordered sequences $LSE(V^{\prime},\alpha)$ (\emph{Lower
Spectral Edge}) and $USE(V^{\prime},\alpha)$ (\emph{Upper Spectral
Edge}) as follows:
\begin{center}
$LSE(V^{\prime},\alpha)=\left(\underline S\left(v_{i_{1}},\alpha
\right),\underline S\left(v_{i_{2}},\alpha \right),\dots,\underline
S\left(v_{i_{k}},\alpha \right)\right)$,
\end{center}
where $\underline S\left(v_{i_{l}},\alpha \right)\leq \underline
S\left(v_{i_{l+1}},\alpha \right)$ for $1\leq l\leq k-1$, and
\begin{center}
$USE(V^{\prime},\alpha)=\left(\overline S\left(v_{j_{1}},\alpha
\right),\overline S\left(v_{j_{2}},\alpha \right),\dots,\overline
S\left(v_{j_{k}},\alpha \right)\right)$,
\end{center}
where $\overline S\left(v_{j_{l}},\alpha \right)\leq \overline
S\left(v_{j_{l+1}},\alpha \right)$ for $1\leq l\leq k-1$.\\

We will use the following results.

\begin{lemma} \cite{TePet}
\label{Hayklemma} If $K_{n,n}$ is a complete bipartite graph with
bipartition $(U,V)$, then for any continuous sequence $L$ with
length $n$, $K_{n,n}$ has a proper edge-coloring $\alpha$ such that:

\begin{description}
\item[1)] for any $u\in U$, $S(u,\alpha)$ is an integer interval;

\item[2)] for any $v\in V$, $S(v,\alpha)$ is an integer interval;

\item[3)] $LSE(U,\alpha)=LSE(V,\alpha)=L$.
\end{description}
\end{lemma}

\begin{lemma} \cite{Kampreprint,KamDiss}
\label{Kn,nmlemma} For any $m,n\in \mathbb{N}$, the complete bipartite graph $K_{n,nm}$ with bipartition $(U,V)$ has an interval $nm$-coloring $\alpha$ such that for any $u\in U$, $S(u,\alpha)=[1,nm]$.
\end{lemma}

\begin{lemma}\cite{PetArXiv}
\label{multipartite} If $G$ is a complete $r$-partite graph with $n$ vertices in each part, then $G\in \mathfrak{N}$ if and only if $nr$ is even. Moreover, if $nr$ is even, then $G$ has an interval $(r-1)n$-coloring $\alpha$ such that for any $v\in V(G)$, $S(v,\alpha)=[1,(r-1)n]$. 
\end{lemma}

\begin{lemma}
\label{lbmultipartite} For any $n_{1},n_{2},\ldots,n_{r}\in
\mathbb{N}$ ($r\geq 3$) with $n_{1}\leq n_{2}\leq \cdots \leq n_{r}$ and $\sum_{i=1}^{r}n_{i}$ is odd,
$$\mathrm{def}\left(K_{n_{1},n_{2},\ldots,n_{r}}\right)\geq \sum_{i=2}^{r}\frac{(n_{1}+1)n_{i}-n_{i}^{2}}{2}.$$
\end{lemma}
\begin{proof}
Clearly, $|V(K_{n_1,n_2,\ldots, n_r})| = \sum_{i=1}^{r}n_{i}$, $|E(K_{n_1,n_2,\ldots, n_r})| =   
\sum \limits_{\substack{1 \leq i < j \leq r}}{n_{i}n_{j}}$ and $\Delta(K_{n_{1}, \ldots,n_{r}}) = \sum_{i=2}^{r}n_{i}$.
In \cite{B-OD-BHal}, it was proved that if $G$ is a graph with an odd number of
vertices, then $\mathrm{def}(G)\geq \frac{2\vert E(G)\vert -(\vert V(G)\vert
-1)\Delta(G)}{2}$. From this and taking into account that $\sum_{i=1}^{r}n_{i}$ is odd, we obtain 
\begin{equation*}
\begin{aligned}
\mathrm{def}\left(K_{n_{1},n_{2},\ldots,n_{r}}\right) &\geq \frac{1}{2} \left (\sum \limits_{\substack{1 \leq i<j \leq r}}{2n_{i}n_{j}} -  (n_{1} + \cdots + n_{r}-1)(n_{2} + \cdots + n_{r})  \right ) = \sum \limits_{\substack{1 \leq i<j \leq r}}{n_{i}n_{j}} -  \\ &- \frac{(n_{1}n_{2} + \ldots n_{1}n_{r} + n_{2}^{2} + \cdots n_{2}n_{r} + \ldots + n_{r}n_{2} + \ldots + n_{r}^{2}) -( n_{2} + \cdots + n_{r})}{2}  = \\ &=
\frac{n_{1}n_{2} + \cdots + n_{1}n_{r} - (n_{2}^{2} + \cdots + n_{r}^{2}) + n_{2} + \cdots + n_{r}}{2} = \sum \limits_{i = 2}^{r}{\frac{(n_{1} + 1)n_{i}-n_{i}^{2}}{2}}.
\end{aligned}
\end{equation*}
\end{proof}

\begin{theorem}
\label{mytheorem1.1}\cite{CasselgrenKhachatrianPetrosyan} If for a graph $G$, there exists a number $d$ such that $d$ divides $d_{G}(v)$ for every $v\in V(G)$ and $d$ does not divide $\vert E(G)\vert$, then $G\notin \mathfrak{N}$.
\end{theorem}

\bigskip

\section{Bounds on the deficiency of complete multipartite graphs}\

In this section we investigate the problem of determining or bounding of the deficiency of complete multipartite graphs. We begin our consideration with an upper bound on the deficiency of complete multipartite graphs.  

\begin{theorem}
\label{general_bound} For any $n_{1},n_{2},\ldots,n_{r}\in
\mathbb{N}$ ($r\geq 3$) with $n_{1}\geq n_{r}\geq n_{2} \geq \cdots \geq n_{r-1}$,
$$\mathrm{def}\left(K_{n_{1},n_{2},\ldots,n_{r}}\right)\leq \sum_{i=2}^{r-1}n_{i}^{2}.$$
\end{theorem}
\begin{proof}
For $0\leq i\leq r$, define a sum $\sigma(i)$ as follows:
\begin{center}
$\sigma(i)=\left\{
\begin{tabular}{ll}
$0$, & if $i=0$,\\
$\sum_{j=1}^{i}n_{j}$, & if $1\leq i\leq r$.\\
\end{tabular}%
\right.$
\end{center}
Let $V_1,\dots,V_r$ be the $r$ independent sets of vertices of
$K_{n_{1},n_{2},\ldots,n_{r}}$, and let
$$V_{i}=\left\{v_{\sigma(i-1)+1},\ldots,v_{\sigma(i)}\right\}$$ for
$1\leq i\leq r$. Also, let $\textbf{s}=\min\{i+j: v_{i}v_{j}\in E\left(
K_{n_{1},n_{2},\ldots,n_{r}}\right)\}$ and $\textbf{S}=\max\{i+j: v_{i}v_{j}\in E\left(K_{n_{1},n_{2},\ldots,n_{r}}\right)\}$.\\

Define an edge-coloring $\alpha$ of $K_{n_{1},n_{2},\ldots,n_{r}}$
as follows: for any $v_{i}v_{j}\in E\left(
K_{n_{1},n_{2},\ldots,n_{r}}\right)$, let

$$\alpha\left(v_{i}v_{j}\right)= 1+i+j - \textbf{s}.$$

Let us prove that $\alpha$ is a proper edge-coloring with $\mathrm{def}\left(K_{n_{1},n_{2},\ldots,n_{r}},\alpha\right)=\sum_{i=2}^{r-1}n_{i}^{2}$.

By the definition, we have that $\alpha$ is a proper edge-coloring with colors $1,2,\ldots,1+\textbf{S}-\textbf{s}$ and for each $v_{i}\in V_{1}$ or $v_{i}\in V_{r}$, $\mathrm{def}(v_{i},\alpha)=0$. Next let $v_{i}\in V_{l}$, where $2\leq l\leq r-1$. By the definition of $\alpha$, we have that $S\left(v_{i},\alpha\right)$
contains colors $i+2-\textbf{s},\ldots,1+i+\sigma(r)-\textbf{s}$ except for $$1+i+\sigma(l-1)+1-\textbf{s},\ldots,1+i+\sigma(l)-\textbf{s}.$$ 
This implies that for each $v_{i}\in V_{l}$ ($2\leq l\leq r-1$), $\mathrm{def}(v_{i},\alpha)=n_{l}$. Thus, $\mathrm{def}\left(K_{n_{1},n_{2},\ldots,n_{r}}\right)\leq \sum_{i=2}^{r-1}n_{i}^{2}$. 
\end{proof}

An example of the coloring $\alpha$ from the proof of Theorem \ref{general_bound} for $n_{1}=4, n_{2}=2, n_{3}=1$ and $n_{4}=3$ one can find in Fig. 1. 

\begin{corollary}
\label{tripartite}
For any $l,m,n\in \mathbb{N}$,
$$\mathrm{def}(K_{l,m,n})\leq \min\{l^{2},m^{2},n^{2}\}.$$
\end{corollary}

Note that the upper bound in Corollary \ref{tripartite} is sharp, since $\mathrm{def}(K_{1,m,n})=1$ if $\gcd(m+1,n+1)\neq 1$ \cite{PetrosHrant}. Next, we consider the deficiency of complete $(r+1)$-partite graphs with $r$ parts with the same size and one large part. 

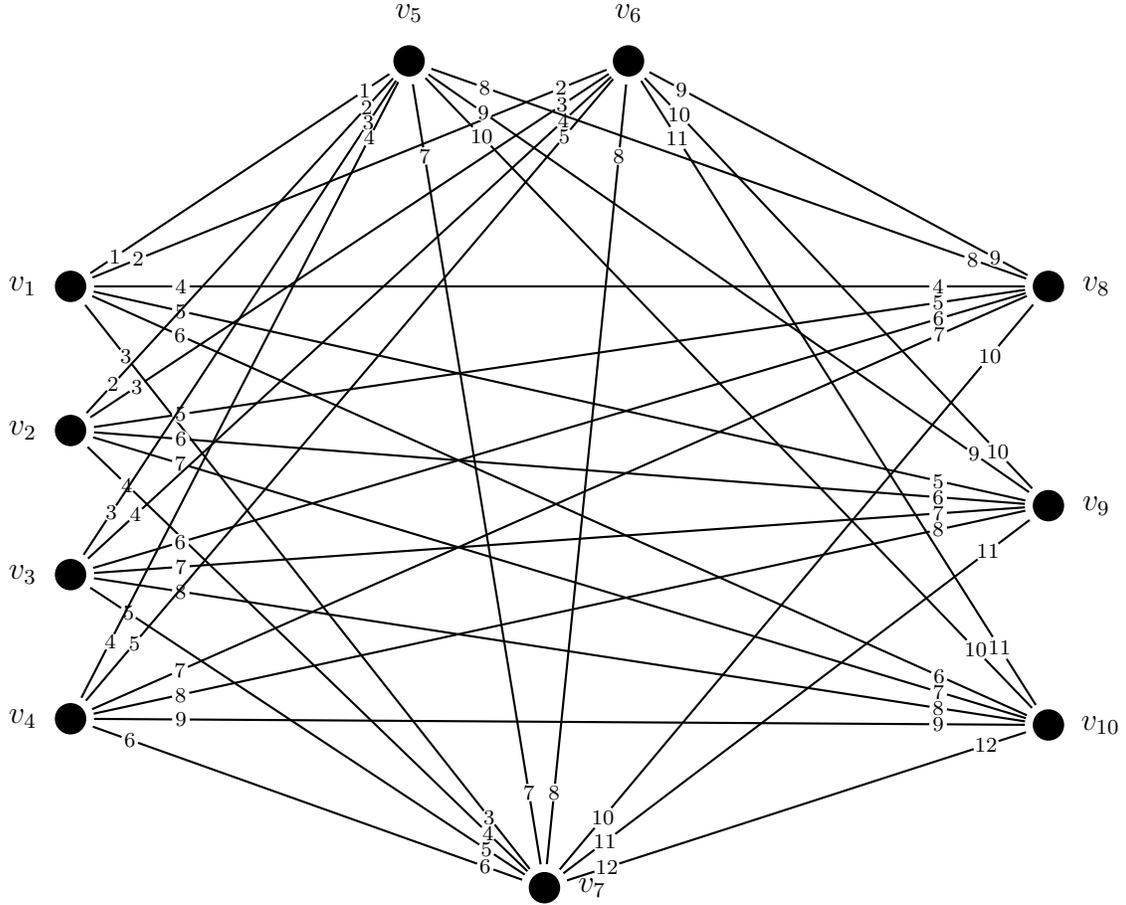
\begin{center}
\begin{tikzpicture}[thick,
  every node/.style={draw,circle, fill=black},
  ->,shorten >= 3pt,shorten <= 3pt
]

\begin{scope}[xshift=0cm,yshift=-3cm, start chain = going below,node distance=15mm]
\foreach \i in {1,...,4}
  \node[on chain] (v\i) [label=left: $v_\i$] {};
\end{scope}

\begin{scope}[xshift=4.5cm,yshift=0cm, start chain=going right,node distance=25mm]
\foreach \i in {5,6}
  \node[on chain] (v\i) [label=above: $v_\i$] {};
\end{scope}

\begin{scope}[xshift=6.3cm,yshift=-11cm,start chain=going right,node distance=25mm]
\foreach \i in {7}
  \node[on chain] (v\i) [label=right: $v_\i$] {};
\end{scope}

\begin{scope}[xshift=13cm,yshift=-3cm,start chain=going below,node distance=25mm]
\foreach \i in {8,9,10}
  \node[on chain] (v\i) [label=right: $v_{\i}$] {};
\end{scope}

\begin{scope}[every node/.style={fill=white, inner sep=0pt,minimum size=1pt, circle,text=black}, every edge/.style={draw=black}]
\foreach \i in {1,...,4}
	\foreach \j in {5,...,10}
		\path[-] (v\i) edge node[pos=0.1, font=\scriptsize] {$\the\numexpr \i+\j-5$} node[pos=0.9, font=\scriptsize] {$\the\numexpr \i+\j-5$}  (v\j);
\foreach \i in {5,6}
	\foreach \j in {7,...,10}
		\path[-] (v\i) edge node[pos=0.1, font=\scriptsize] {$\the\numexpr \i+\j-5$} node[pos=0.9, font=\scriptsize] {$\the\numexpr \i+\j-5$}  (v\j);
\foreach \i in {7}
	\foreach \j in {8,...,10}
		\path[-] (v\i) edge node[pos=0.1, font=\scriptsize] {$\the\numexpr \i+\j-5$} node[pos=0.9, font=\scriptsize] {$\the\numexpr \i+\j-5$}  (v\j);
\end{scope}

\end{tikzpicture}\\
The proper edge-coloring $\alpha$ with $12$ colors of $K_{4,2,1,3}$ with $\mathrm{def}(K_{4,2,1,3},\alpha)=5$.
\end{center}

\begin{theorem}
\label{multi_trn}
For any $n,r,t \in \mathbb{N}$, if $nr$ is even, then for the complete $(r+1)$-partite graph $K_{n,\ldots,n,trn}$, $$\mathrm{def}(K_{n,\ldots,n,trn}) = 0.$$
\end{theorem}
\begin{proof}
Let $G=K_{n,\ldots,n,trn}$ be the complete $(r+1)$-partite graph and $V(G)=\bigcup\limits_{i=1}^{r} V_{i} \cup W$, where $V_1, V_2, \ldots, V_r$ and $W$ are parts of $G$ with sizes $|V_i|=n (1 \le i \le r)$ and $|W| = trn$.
\par
Let $H_1 = G \Big[\bigcup\limits_{i=1}^{r} V_{i} \Big]$ and $H_2 = \Big(\bigcup\limits_{i=1}^{r} G[V_{i}]\Big)  \lor G[W]$. Clearly, $E(H_1) \cap E(H_2) = \varnothing$ and $E(H_1) \cup E(H_2) = E(G)$. Moreover, $H_1$ is isomorphic to the complete $r$-partite graph $K_{n,\ldots, n}$ and $H_2$ is isomorphic to the complete bipartite graph $K_{rn, trn}$.
\par
By Lemma \ref{multipartite}, $H_1$ has an interval $(r-1)n$-coloring $\alpha$ such that for any $v \in V(H_1)$, $S_{H_1}(v,\alpha) = [1, (r-1)n]$. By Lemma \ref{Kn,nmlemma}, $H_2$ has an interval $trn$-coloring $\beta$ such that for any $u \in \bigcup\limits_{i=1}^r V_i$, $S_{H_2}(u, \beta) = [1, trn]$. Now we define an edge-coloring $\gamma$ as follows: for any $e \in E(G)$, let
\begin{equation*}
  \gamma(e)=\begin{cases}
    \alpha(e), & \text{if $e \in E(H_1)$},\\
    \beta(e) + (r-1)n, & \text{if $e \in E(H_2)$}.
  \end{cases}
\end{equation*}

Let us show that $\gamma$ is an interval $(tr + r - 1)n$-coloring.\\

By the definition of $\gamma$, we have

\begin{description}
\item[1)] for any $v \in \bigcup\limits_{i=1}^r V_i$,
\begin{equation*}
S_G(v, \gamma) = S_{H_1}(v, \alpha) \cup S_{H_2}(v, \beta)=[1,(r-1)n] \cup [(r-1)n +1, (r-1)n + trn] = [1,(tr +r-1)n].    
\end{equation*}

\item[2)] for any $v \in W$,
$$S_G(v, \gamma) = S_{H_2}(v, \beta) \oplus (r-1)n = [\underline{S}_{H_2}(v, \beta) + (r-1)n, \overline{S}_{H_2}(v, \beta) + (r-1)n].$$
\end{description}

This shows that $\gamma$ is an interval $(tr +r-1)n$-coloring of $G$; thus $\mathrm{def}(K_{n,\ldots,n,trn}) = 0$. 
\end{proof}

An example of the coloring $\gamma$ from the proof of Theorem \ref{multi_trn} for $t=1, r=2$ and $n=3$ one can find in Fig. 2.

\begin{center}
\begin{tikzpicture}[thick,
  every node/.style={draw,circle},
  colornode/.style={draw},
  fsnode/.style={fill=black},
  ssnode/.style={fill=black},
  every fit/.style={ellipse,draw,inner sep=-2pt,text width=2cm},
  ->,shorten >= 3pt,shorten <= 3pt,
]

\begin{scope}[xshift=1.85cm,start chain = going right,node distance=14mm]
\foreach \i in {1,...,6}
  \node[fsnode,on chain] (w\i) [label=above: $w_\i$] {};
\end{scope}

\begin{scope}[xshift=0cm,yshift=-4cm,start chain=going below,node distance=25mm]
\foreach \i in {1,2,3}
  \node[fsnode,on chain] (u\i) [label=left: $u_\i$] {};
\end{scope}

\begin{scope}[xshift=13cm,yshift=-4cm,start chain=going below,node distance=25mm]
\foreach \i in {1,2,3}
  \node[ssnode,on chain] (v\i) [label=right: $v_\i$] {};
\end{scope}

\begin{scope}[every node/.style={fill=white, inner sep=0pt,minimum size=1pt, circle,text=black}, every edge/.style={draw=black}]
\foreach \i in {1,2,3}
	\foreach \j in {1,2,3}
		\path[-] (v\i) edge node[near start, font=\scriptsize] {$\the\numexpr \intcalcMod{\i+\j }{3}+1\relax $} node[near end, font=\scriptsize] {$\the\numexpr \intcalcMod{\i+\j }{3}+1\relax $} (u\j);
\foreach \i in {1,2,3}
	\foreach \j in {1,2,3}
		\path[-] (w\i) edge node[pos=0.15,xshift=0mm, font=\scriptsize] {$\the\numexpr \intcalcMod{\i+\j }{3}+1 +3\relax $} node[near end,xshift=0mm, font=\scriptsize] {$\the\numexpr \intcalcMod{\i+\j }{3}+1 +3\relax $}(u\j);
\foreach \i in {4,5,6}
	\foreach \j in {1,2,3}
		\path[-] (w\i) edge node[pos=0.15,xshift=0mm, font=\scriptsize] {$\the\numexpr \intcalcMod{\i-3+\j }{3}+1 +6\relax $} node[near end,xshift=0mm, font=\scriptsize] {$\the\numexpr \intcalcMod{\i-3+\j }{3}+1 +6\relax $} (u\j);
\foreach \i in {1,2,3}
	\foreach \j in {1,2,3}
		\path[-] (w\i) edge node[pos=0.15,xshift=0mm, font=\scriptsize] {$\the\numexpr \intcalcMod{\i+\j }{3}+1 +6\relax $} node[near end,xshift=0mm, font=\scriptsize] {$\the\numexpr \intcalcMod{\i+\j }{3}+1 +6\relax $}  (v\j);
\foreach \i in {4,5,6}
	\foreach \j in {1,2,3}
		\path[-] (w\i) edge node[pos=0.15,xshift=0mm, font=\scriptsize] {$\the\numexpr \intcalcMod{\i-3+\j }{3}+1 +3\relax $} node[near end,xshift=0mm, font=\scriptsize] {$\the\numexpr \intcalcMod{\i-3+\j }{3}+1 +3\relax $}(v\j);
		
\end{scope}
\end{tikzpicture}\\
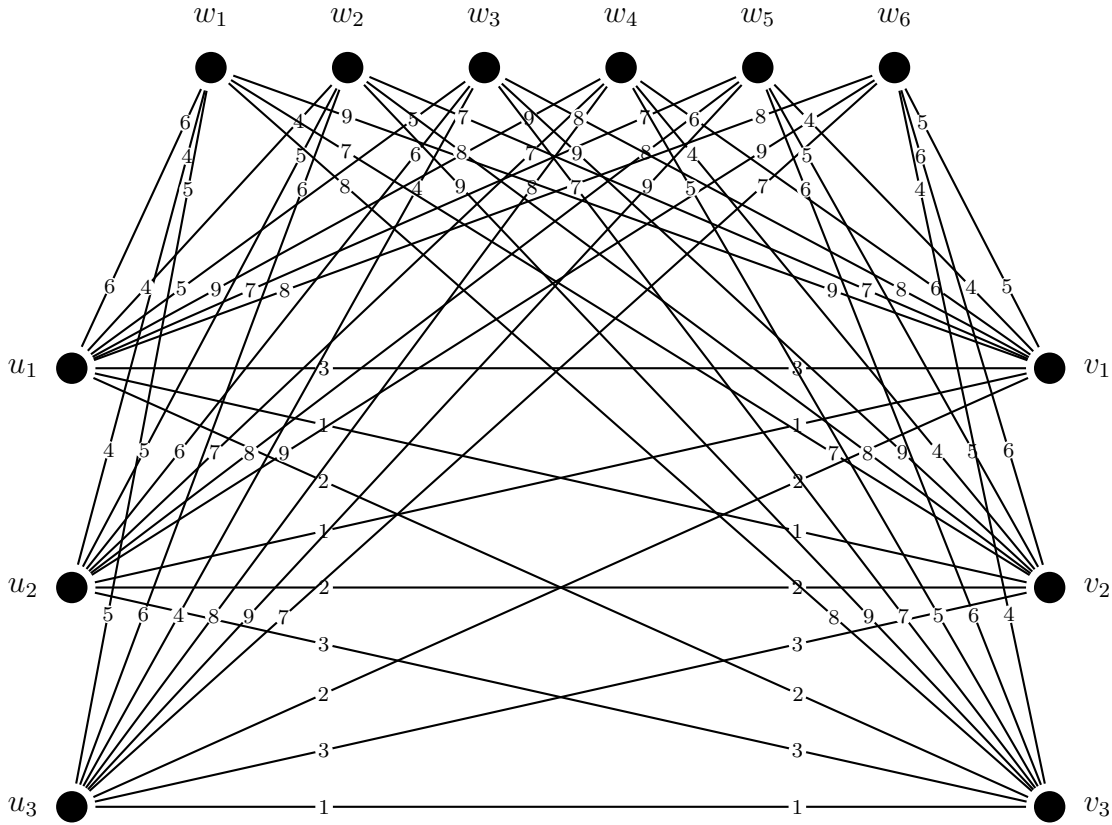
The interval $9$-coloring $\gamma$ of $K_{3,3,6}$.
\end{center}

\begin{theorem}
\label{multi_(tr+1)n}
For any $n,r,t \in \mathbb{N}$, for the complete $(r+1)$-partite graph $K_{n,\ldots,n,(tr+1)n}$, $\mathrm{def}(K_{n,\ldots,n,(tr+1)n}) = 0$
if and only if $n(r+1)$ is even.
\end{theorem}
\begin{proof}
Let $G=K_{n,\ldots,n,(tr+1)n}$ be the complete $(r+1)$-partite graph.
\par
The necessity we prove by contradiction. Assume that there are natural numbers $t,r$ and $n$ such that $\mathrm{def}(G) = 0$ and $n(r+1)$ is odd. First we calculate the number of edges of $G$:
$$|E(G)| = nr(tr+1)n+\frac{r(r-1)n^{2}}{2} = rn^{2} \Big( tr + 1 +  \frac{r-1}{2} \Big).$$
Let $d$ be the greatest common divisor of degrees of vertices of $G$. Clearly for any $v \in V(G)$ either $d_G(v) = nr$ or $d_G(v)=n(r-1) + (rt+1)n$. From here, we obtain
$$d = \gcd(nr, nr + nrt) = 
nr\gcd(1, 1+t) = nr.$$

By Theorem \ref{mytheorem1.1}, we have $|E(G)|\equiv 0 \mod d$, which means the number $n(tr + 1 +  \frac{r-1}{2})$ must be integer, which is a contradiction,  since $n(r+1)$ is odd.
\par 
For the proof of sufficiency, let 
$V(G)=\bigcup\limits_{i=1}^{r} V_{i} \cup U \cup W$, 
where $V_1, V_2, \ldots, V_r$ and $U \cup W$ are parts of $G$ with sizes $|V_i|=n$ $(1 \le i \le r)$, $|U| = n$ and $|W| = trn$.
\par
Let $H_1 = G \Big[\bigcup\limits_{i=1}^{r} V_{i} \cup U\Big]$ and $H_2 = \Big(\bigcup\limits_{i=1}^{r} G[V_{i}]\Big)  \lor G[W]$. Clearly, $E(H_1) \cap E(H_2) = \varnothing$ and $E(H_1) \cup E(H_2) = E(G)$. Moreover, $H_1$ is isomorphic to the complete $(r+1)$-partite graph $K_{n,\ldots, n}$ and $H_2$ is isomorphic to the complete bipartite graph $K_{rn, trn}$.
\par
By Lemma \ref{multipartite}, $H_1$ has an interval $rn$-coloring $\alpha$ such that for any $v \in V(H_1)$ $S_{H_1}(v,\alpha) = [1, rn]$. By Lemma \ref{Kn,nmlemma}, $H_2$ has an interval $trn$-coloring $\beta$ such that for any $u \in \bigcup\limits_{i=1}^r V_i$, $S_{H_2}(u, \beta) = [1, trn]$. Now we define an edge-coloring $\gamma$ as follows: for any $e \in E(G)$, let
\begin{equation*}
  \gamma(e)=\begin{cases}
    \alpha(e), & \text{if $e \in E(H_1)$},\\
    \beta(e) + rn, & \text{if $e \in E(H_2)$}.
  \end{cases}
\end{equation*}

Let us show that $\gamma$ is an interval $(t+1)rn$-coloring. 

By the definition of $\gamma$, we have
\par
\begin{description}
\item[1)] for any $v \in \bigcup\limits_{i=1}^r V_i$,
$$S_G(v, \gamma) = S_{H_1}(v, \alpha) \cup S_{H_2}(v, \beta) 
= [1,rn] \cup [rn +1, rn + trn] = [1,(t+1)rn],$$

\item[2)] for any $v \in W$,
$$S_G(v, \gamma) = S_{H_2}(v, \beta) \oplus rn = [\underline{S}_{H_2}(v, \beta) + rn, \overline{S}_{H_2}(v, \beta) + rn].$$
\end{description}

This shows that $\gamma$ is an interval $(t+1)rn$-coloring of $G$; thus $\mathrm{def}(K_{n,\ldots,n,(tr+1)n}) = 0$.
\end{proof}

An example of the coloring $\gamma$ from the proof of Theorem \ref{multi_(tr+1)n} for $t=1, r=2$ and $n=2$ one can find in Fig. 3.

\begin{center}
\begin{tikzpicture}[thick,
  every node/.style={draw,circle},
  colornode/.style={draw},
  fsnode/.style={fill=black},
  ssnode/.style={fill=black},
  every fit/.style={ellipse,draw,inner sep=-2pt,text width=2cm},
  ->,shorten >= 3pt,shorten <= 3pt,
]

\begin{scope}[xshift=1.8cm,start chain = going right,node distance=14mm]
\foreach \i in {1,...,6}
  \node[fsnode,on chain] (w\i) [label=above: $w_\i$] {};
\end{scope}

\begin{scope}[xshift=0cm,yshift=-4cm,start chain=going below,node distance=33mm]
\foreach \i in {1,2}
  \node[fsnode,on chain] (u\i) [label=left: $u_\i$] {};
\end{scope}

\begin{scope}[xshift=13cm,yshift=-4cm,start chain=going below,node distance=33mm]
\foreach \i in {1,2}
  \node[ssnode,on chain] (v\i) [label=right: $v_\i$] {};
\end{scope}

\begin{scope}[every node/.style={fill=white, inner sep=0pt,minimum size=1pt, circle,text=black}, every edge/.style={draw=black}]

\path[-] (w1) edge node[pos=0.15,xshift=0mm, font=\scriptsize] {$3$} node[near end,xshift=0mm, font=\scriptsize] {$3$}(u1);
\path[-] (w2) edge node[pos=0.15,xshift=0mm, font=\scriptsize] {$1$} node[near end,xshift=0mm, font=\scriptsize] {$1$}(u1);
\path[-] (v1) edge node[pos=0.15,xshift=0mm, font=\scriptsize] {$4$} node[near end,xshift=0mm, font=\scriptsize] {$4$}(u1);
\path[-] (v2) edge node[pos=0.15,xshift=0mm, font=\scriptsize] {$2$} node[near end,xshift=0mm, font=\scriptsize] {$2$}(u1);

\path[-] (w1) edge node[pos=0.15,xshift=0mm, font=\scriptsize] {$4$} node[near end,xshift=0mm, font=\scriptsize] {$4$}(u2);
\path[-] (w2) edge node[pos=0.15,xshift=0mm, font=\scriptsize] {$2$} node[near end,xshift=0mm, font=\scriptsize] {$2$}(u2);
\path[-] (v1) edge node[pos=0.15,xshift=0mm, font=\scriptsize] {$1$} node[near end,xshift=0mm, font=\scriptsize] {$1$}(u2);
\path[-] (v2) edge node[pos=0.15,xshift=0mm, font=\scriptsize] {$3$} node[near end,xshift=0mm, font=\scriptsize] {$3$}(u2);

\path[-] (w1) edge node[pos=0.15,xshift=0mm, font=\scriptsize] {$2$} node[near end,xshift=0mm, font=\scriptsize] {$2$}(v1);
\path[-] (w1) edge node[pos=0.15,xshift=0mm, font=\scriptsize] {$1$} node[near end,xshift=0mm, font=\scriptsize] {$1$}(v2);
\path[-] (w2) edge node[pos=0.15,xshift=0mm, font=\scriptsize] {$3$} node[near end,xshift=0mm, font=\scriptsize] {$3$}(v1);
\path[-] (w2) edge node[pos=0.15,xshift=0mm, font=\scriptsize] {$4$} node[near end,xshift=0mm, font=\scriptsize] {$4$}(v2);

\foreach \i in {3,4}
	\foreach \j in {1,2}
		\path[-] (w\i) edge node[pos=0.15,xshift=0mm, font=\scriptsize] {$\the\numexpr \intcalcMod{\i-2+\j }{2}+1 +4\relax $} node[near end,xshift=0mm, font=\scriptsize] {$\the\numexpr \intcalcMod{\i-2+\j }{2}+1 +4\relax $}(u\j);
\foreach \i in {5,6}
	\foreach \j in {1,2}
		\path[-] (w\i) edge node[pos=0.15,xshift=0mm, font=\scriptsize] {$\the\numexpr \intcalcMod{\i-4+\j }{2}+1 +6\relax $} node[near end,xshift=0mm, font=\scriptsize] {$\the\numexpr \intcalcMod{\i-4+\j }{2}+1 +6\relax $} (u\j);

\foreach \i in {3,4}
	\foreach \j in {1,2}
		\path[-] (w\i) edge node[pos=0.15,xshift=0mm, font=\scriptsize] {$\the\numexpr \intcalcMod{\i-2+\j }{2}+1 +6\relax $} node[near end,xshift=0mm, font=\scriptsize] {$\the\numexpr \intcalcMod{\i-2+\j }{2}+1 +6\relax $}(v\j);
\foreach \i in {5,6}
	\foreach \j in {1,2}
		\path[-] (w\i) edge node[pos=0.15,xshift=0mm, font=\scriptsize] {$\the\numexpr \intcalcMod{\i-4+\j }{2}+1 +4\relax $} node[near end,xshift=0mm, font=\scriptsize] {$\the\numexpr \intcalcMod{\i-4+\j }{2}+1 +4\relax $} (v\j);

\end{scope}
\end{tikzpicture}\\
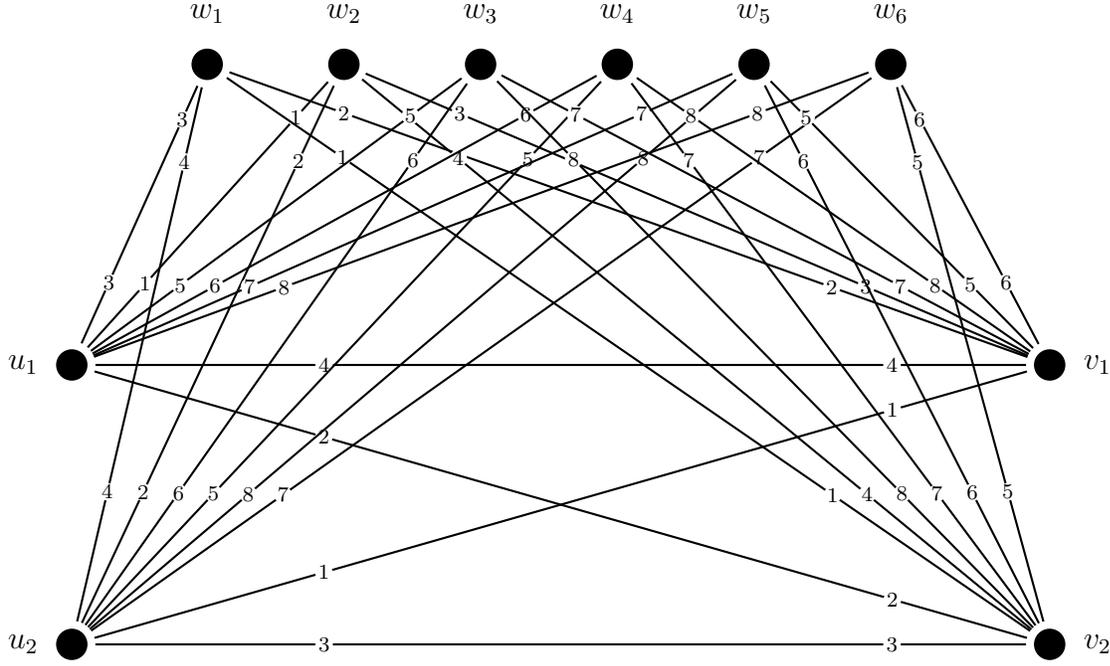
The interval $8$-coloring $\gamma$ of $K_{2,2,6}$.
\end{center}

Now we consider the deficiency of complete multipartite graphs which are close to balanced complete multipartite graphs. For such graphs we prove lower and upper bounds on the deficiency which are differ from each other by one.  

\begin{theorem}
For any $n,r \in \mathbb{N}$, if $n(r+1)$ is even, then for the complete $(r+1)$-partite graph $K_{n,\ldots,n,n+1}$,
$$\frac{n(r-1)}{2} \le \mathrm{def}(K_{n,\ldots,n,n+1}) \le \frac{n(r-1)}{2}+1.$$
\end{theorem}

\begin{proof}
Let $G=K_{n,\ldots,n,n+1}$ be the complete $(r+1)$-partite graph.

First we prove the lower bound.  Since $|V(G)| = n(r+1) +1$ is odd, by Lemma \ref{lbmultipartite}, we obtain
$$\mathrm{def}(G) \ge (r - 1) \frac{(n+1)  n - n^2}{2} + \frac{(n+1)(n+1) - (n+1)^2}{2} = \frac{n(r - 1)}{2}.$$
Next we show that $\mathrm{def}(G) \le \frac{n(r - 1)}{2}+1$. We distinguish this part of the proof into two cases.\\
\textbf{Case 1:} $r$ is odd.\\
Let $V(G) = \left(\bigcup\limits_{i=1}^{r+1} V_i \right) \cup \{w\}$, where $V_1, V_2, \ldots, V_r$ and $V_{r+1} \cup \{w\}$ are parts of $G$ with sizes $|V_i|=n$ $(1 \le i \le r)$ and $|V_{r+1} \cup \{w\}| = n+1$. Also, let $V_i=\left\{v_j^{(i)}: 1 \le j \le n\right\}$ $(1 \le i \le r+1)$ and $H=G\left[ \bigcup\limits_{i=1}^{r+1}V_i \right]$.
Clearly, $H$ is isomorphic to the complete $(r+1)$-partite graph $K_{n, \ldots,n}$.
In \cite{PetArXiv}, it was proved that $H$ has an interval 
$\big((3r+1)\frac{n}{2}-1\big)$-coloring $\alpha$ such that for any $1 \le i \le \frac{r+1}{2}$ and $1 \le j \le n$,
\begin{gather*}
S\left(v_j^{(2i - 1)},\alpha\right) = S\left(v_j^{(2i)},\alpha\right) = [j + (i-1)n, j + (i-1)n + rn -1] = \\
=[j + (i - 1)n, j + (r+i-1)n -1].
\end{gather*}

Define an edge-coloring $\beta$ of $G$ as follows: for any $e \in E(G)$, let
\begin{equation*}
  \beta(e)=\begin{cases}
    \alpha(e)+1,      &\text{if } e \in E(H),\\
    (i-1)n+j,           &\text{if } e=wv^{(2i-1)}_j, 1 \le i \le \frac{r+1}{2}, 1 \le j \le n,\\
    (r+i-1)n+j+1,     &\text{if } e=wv^{(2i)}_j,   1 \le i < \frac{r+1}{2},   1 \le j \le n.
  \end{cases}
\end{equation*}
Let us prove that $\beta$ is a proper edge-coloring of $G$ with $\frac{(3r+1)n}{2}$ colors such that $\mathrm{def}(G, \beta) = \mathrm{def}\left(w,\beta\right) = \frac{n(r-1)}{2}+1$. 

By the definition of $\beta$, we have

\begin{description}
\item[1)] for $1 \le i \le \frac{r+1}{2}, 1 \le j \le n,$ 
$$S\left(v_j^{(2i-1)},\beta\right) = [(i-1)n+j, (i+r-1)n+j],$$
\item[2)] for $1 \le i < \frac{r+1}{2}, 1 \le j \le n,$
$$S\left(v_j^{(2i)},\beta\right) = [(i-1)n+j+1, (i+r-1)n+j+1],$$
\item[3)] for $i = \frac{r+1}{2}, 1 \le j \le n,$
$$S\left(v_j^{(2i)},\beta\right) = [(i-1)n+j+1, (i+r-1)n+j],$$
\item[4)] $S\left(w,\beta\right) = 
\bigcup\limits_{i=1}^{\frac{r+1}{2}} \bigcup\limits_{j=1}^{n} \{(i-1)n+j \} \cup 
\bigcup\limits_{i=1}^{\frac{r-1}{2}} \bigcup\limits_{j=1}^{n} \{(r+i-1)n+j+1\} = $
$\bigcup\limits_{i=1}^{\frac{r+1}{2}} [(i-1)n+1,in]\cup\bigcup\limits_{i=1}^{\frac{r-1}{2}} [(r+i-1)n+2,(r+i)n+1]=$
$\left[1, \frac{r+1}{2}n\right]\cup\left[rn+2,\left(r+\frac{r-1}{2}\right)n+1\right].$
\end{description}

This shows that for any $v \in V(G)\setminus \{w\}$, $S(v, \beta)$ is an integer interval. Now we consider the spectrums of the vertex $w$. By $4)$, we obtain
\begin{description}
\item[a)] $|S(w,\beta)| = \frac{r+1}{2}n + (r+ \frac{r-1}{2})n -  rn = rn = d_G(w),$
\item[b)] $\mathrm{def}(w,\beta) = \overline{S}(w,\beta) - \underline{S}(w,\beta) - |S(w,\beta)| + 1 =(r+ \frac{r-1}{2})n - rn  + 1 =  \frac{r-1}{2}n + 1.$
\end{description}

This implies that $\beta$ is a proper edge-coloring of $G$ with $\frac{(3r+1)n}{2}$ colors such that $\mathrm{def}(G, \beta) = \mathrm{def}(w,\beta) = \frac{n(r-1)}{2}+1$.\\

\textbf{Case 2:} $n$ is even.\\

Let $V(G) = \left(\bigcup\limits_{i=1}^{2(r+1)} V_i \right) \cup \{w\}$, where $V_1 \cup V_3$, $V_{2i}\cup V_{2i+3}$ $(1\le i\le r-1)$ and $V_{2r} \cup V_{2(r+1)}\cup \{w\}$ are parts of $G$ with sizes $|V_i|=\frac{n}{2}$ $(1 \le i \le 2(r+1))$. Also, let $V_i=\left\{v_j^{(i)}: 1 \le j \le \frac{n}{2}\right\}$  $(1 \le i \le 2(r+1))$ and
$H=G\left[ \bigcup\limits_{i=1}^{2(r+1)}V_i \right]$.
Clearly, $H$ is isomorphic to the complete $(r+1)$-partite graph $K_{n, \ldots,n}$. In \cite{PetArXiv}, it was proved that $H$ has an interval 
$\left(\frac{(3r+1)n}{2}-1\right)$-coloring $\gamma$ such that for any $1 \le i \le r+1$ and $1 \le j \le \frac{n}{2}$,
$$S\left(v_j^{(2i - 1)},\gamma\right) = S\left(v_j^{(2i)},\gamma\right) =\left[j + (i-1)\frac{n}{2}, j + (2r+i-1)\frac{n}{2}-1\right].$$

Define an edge-coloring $\varphi$ of $G$ as follows: for any $e \in E(G)$, let
\begin{equation*}
  \varphi(e)=\begin{cases}
    \gamma(e) + 1,  &\text{if } e \in E(H),\\
    (i-1)\frac{n}{2}+j,       &\text{if } e=wv^{(2i-1)}_j, 1 \le i \le r+1,   1 \le j \le \frac{n}{2},\\
    (2r+i-1)\frac{n}{2}+j+1,&\text{if } e=wv^{(2i)}_j,   1 \le i \le r-1,   1 \le j \le \frac{n}{2}.
  \end{cases}
\end{equation*}

Let us prove that $\varphi$ is a proper edge-coloring of $G$ with $\frac{(3r+1)n}{2}$ colors such that $\mathrm{def}(G,\varphi) = \mathrm{def}(w,\varphi) = \frac{n(r-1)}{2}+1$. 

By the definition of $\varphi$, we have

\begin{description}
\item[1')] for $1 \le i \le r+1, 1 \le j \le \frac{n}{2},$ 
$$S\left(v_j^{(2i-1)},\varphi\right) = \left[j + (i - 1)\frac{n}{2}, j + (i + 2r- 1)\frac{n}{2}\right],$$
\item[2')] for $1 \le i \le r-1, 1 \le j \le \frac{n}{2},$
$$S\left(v_j^{(2i)},\varphi\right) = \left[j + (i - 1)\frac{n}{2} + 1, j + (i + 2r- 1)\frac{n}{2}+1\right],$$
\item[3')] for $i \in \{ r, r+1 \}, 1 \le j \le \frac{n}{2},$
$$S\left(v_j^{(2i)},\varphi\right) = \left[j + (i - 1) \frac{n}{2}+1, j + (i + 2r- 1) \frac{n}{2}\right],$$
\item[4')] $S(w,\varphi) = 
\bigcup\limits_{i=1}^{r+1} \bigcup\limits_{j=1}^{\frac{n}{2}} \{ j + (i-1)\frac{n}{2} \} \cup 
\bigcup\limits_{i=1}^{r-1} \bigcup\limits_{j=1}^{\frac{n}{2}} \{ j + (2r+i-1)\frac{n}{2} +1\} = $
$\bigcup\limits_{i=1}^{r+1} \left[(i-1)\frac{n}{2}+1,i\frac{n}{2}\right]\cup\bigcup\limits_{i=1}^{r-1} \left[(2r+i-1)\frac{n}{2}+2,(2r+i)\frac{n}{2}+1\right]=$
$\left[1, \frac{n(r+1)}{2}\right]\cup\left[nr+2,\frac{(3r-1)n}{2}+1\right].$
\end{description}

This shows that for any $v \in V(G)\setminus \{w\}$, $S(v, \varphi)$ is an integer interval. Now we consider the spectrums of the vertex $w$. By $4')$, we obtain
\begin{description}
\item[a')] $|S(w,\varphi)| = \frac{n(r+1)}{2} + \frac{n(3r-1)}{2} -  nr = nr = d_G(w),$
\item[b')] $\mathrm{def}(w,\varphi) = \overline{S}(w,\varphi) - \underline{S}(w,\varphi) - |S(w,\varphi)| + 1 =
\frac{n(3r-1)}{2} - nr  + 1 =  \frac{n(r-1)}{2} + 1.$
\end{description}

This implies that $\varphi$ is a proper edge-coloring of $G$ with $\frac{(3r+1)n}{2}$ colors such that $\mathrm{def}(G, \varphi) = \mathrm{def}(w,\varphi) = \frac{n(r-1)}{2}+1$.
\end{proof}

Next, we consider the deficiency of complete $(r+1)$-partite graphs with $r$ parts with the same size $n$ and one part with size $n+2$. For such graphs we prove an upper bound on the deficiency.  

\begin{theorem}
For any $n,r \in \mathbb{N}$, if $n(r+1)$ is even, then for the complete $(r+1)$-partite graph $K_{n,\ldots,n,n+2}$,
$$\mathrm{def}(K_{n,\ldots,n,n+2}) \le nr+2.$$
\end{theorem}

\begin{proof}
Let $G=K_{n,\ldots,n,n+2}$ be the complete $(r+1)$-partite graph.
We distinguish our proof into two cases.\\
\textbf{Case 1:} $r$ is odd.\\
Let $V(G) = \left(\bigcup\limits_{i=1}^{r+1} V_i \right) \cup \{w_1, w_2\}$, where $V_1, V_2, \ldots, V_r$ and $V_{r+1} \cup \{w_1, w_2\}$ are parts of $G$ with sizes $|V_i|=n$ $(1 \le i \le r)$ and $|V_{r+1} \cup \{w_1, w_2\}| = n+2$. Also, let $H=G\left[ \bigcup\limits_{i=1}^{r+1}V_i \right]$ and $V_i=\left\{v_j^{(i)}: 1 \le j \le n\right\}$ $(1 \le i \le r+1)$.
Clearly, $H$ is isomorphic to the complete $(r+1)$-partite graph $K_{n, \ldots,n}$.
In \cite{PetArXiv}, it was proved that $H$ has an interval 
$\big(\frac{n(3r+1)}{2}-1\big)$-coloring $\alpha$ such that for any $1 \le i \le \frac{r+1}{2}$ and $1 \le j \le n$,
\begin{gather*}
S \left(v_j^{(2i - 1)},\alpha \right) = S \left(v_j^{(2i)},\alpha \right) = [j + (i-1)n, j + (i-1)n + rn -1] = \\
= [j + (i - 1)n, j + (i+r-1)n -1].
\end{gather*}

Define an edge-coloring $\beta$ of $G$ as follows: for any $e \in E(G)$, let
\begin{equation*}
  \beta(e)=\begin{cases}
    \alpha(e) + 1,      &\text{if } e \in E(H),\\
    j+(i-1)n,           &\text{if } e=w_1v^{(2i-1)}_j, 1 \le i \le \frac{r+1}{2}, 1 \le j \le n,\\
    j+(r+i-1)n + 1,     &\text{if } e=w_1v^{(2i)}_j,   1 \le i < \frac{r+1}{2},   1 \le j \le n,\\
    j+(r+i-1)n + 1,     &\text{if } e=w_2v^{(2i-1)}_j, 1 \le i \le \frac{r+1}{2}, 1 \le j \le n,\\
    j+(i-1)n,           &\text{if } e=w_2v^{(2i)}_j,   1 \le i < \frac{r+1}{2},   1 \le j \le n.
  \end{cases}
\end{equation*}
Let us prove that $\beta$ is a proper edge-coloring of $G$ with $\frac{n(3r+1)}{2}+1$ colors such that $\mathrm{def}(G, \beta) = \mathrm{def}(w_1,\beta)+\mathrm{def}(w_2,\beta) = nr+2$. 

By the definition of $\beta$, we have

\begin{description}
\item[1)] for $1 \le i \le \frac{r+1}{2}, 1 \le j \le n,$ 
$$S \left(v_j^{(2i-1)},\beta \right) = [j + (i - 1)n, j + (i + r- 1)n+1],$$
\item[2)] for $1 \le i < \frac{r+1}{2}, 1 \le j \le n,$
$$S \left(v_j^{(2i)},\beta \right) = [j + (i - 1)n, j + (i + r- 1)n+1],$$
\item[3)] for $i = \frac{r+1}{2}, 1 \le j \le n,$
$$S \left(v_j^{(2i)},\beta \right) = [j + (i - 1) n+1, j + (i + r- 1) n],$$
\item[4)] $S \left(w_1,\beta \right) = 
\bigcup\limits_{i=1}^{\frac{r+1}{2}} \bigcup\limits_{j=1}^{n} \{ j + (i-1)n \} \cup 
\bigcup\limits_{i=1}^{\frac{r-1}{2}} \bigcup\limits_{j=1}^{n} \{ j + (r+i-1)n +1\} = $
$\bigcup\limits_{i=1}^{\frac{r+1}{2}} \left[(i-1)n+1,in\right]
\cup\bigcup\limits_{i=1}^{\frac{r-1}{2}} \left[(r+i-1)n+2,(r+i)n+1\right]=$
$\left[1, \frac{r+1}{2}n\right]\cup\left[2+ rn,1+\left(r+\frac{r-1}{2}\right)n\right],$
\item[5)] $S \left(w_2,\beta \right) = 
\bigcup\limits_{i=1}^{\frac{r-1}{2}} \bigcup\limits_{j=1}^{n} \{ j + (i-1)n \} \cup 
\bigcup\limits_{i=1}^{\frac{r+1}{2}} \bigcup\limits_{j=1}^{n} \{ j + (r+i-1)n +1\} = $
$\bigcup\limits_{i=1}^{\frac{r-1}{2}} \left[(i-1)n+1,in\right]\cup\bigcup\limits_{i=1}^{\frac{r+1}{2}} \left[(r+i-1)n+2,(r+i)n+1\right]=$
$\left[1, \frac{r-1}{2}n\right]\cup\left[2+ rn,1+\left(r+\frac{r+1}{2}\right)n\right].$
\end{description}

This shows that for any $v \in V(G)\setminus \{w_1, w_2\}$, $S(v, \beta)$ is an integer interval. Now we consider the spectrums of the vertices $w_1$ and $w_2$. By $4)$ and $5)$, we obtain
\begin{description}
\item[a)]$|S(w_1,\beta)| = |S(w_2,\beta)| = rn = d_G(w_1) = d_G(w_2),$
\item[b)] $\mathrm{def}(w_1,\beta) = \overline{S}(w_1,\beta) - \underline{S}(w_1,\beta) - |S(w_1,\beta)| + 1 = 
(r+ \frac{r-1}{2})n - rn  + 1 =  \frac{r-1}{2}n + 1,$
\item[c)] $\mathrm{def}(w_2,\beta) = \overline{S}(w_2,\beta) - \underline{S}(w_2,\beta) - |S(w_2,\beta)| + 1 = 
(r+ \frac{r+1}{2})n - rn  + 1 =  \frac{r+1}{2}n + 1.$
\end{description}

This implies that $\beta$ is a proper edge-coloring of $G$ with $\frac{n(3r+1)}{2}+1$ colors such that $\mathrm{def}(G, \beta) = \mathrm{def}(w_1,\beta)+\mathrm{def}(w_2,\beta) = nr+2$.\\

\textbf{Case 2:} $n$ is even.\\

Let $V(G) = \left(\bigcup\limits_{i=1}^{2(r+1)} V_i \right) \cup \{w_1,w_2\}$, where $V_1 \cup V_3$, $V_{2i}\cup V_{2i+3}$ $(1\le i\le r-1)$ and $V_{2r} \cup V_{2(r+1)}\cup \{w_1, w_2\}$ are parts of $G$ with sizes $|V_i|=\frac{n}{2}$ $(1 \le i \le 2(r+1))$. Also, let $H=G\left[ \bigcup\limits_{i=1}^{2(r+1)}V_i \right]$ and $V_i=\left\{v_j^{(i)}: 1 \le j \le \frac{n}{2}\right\}$ $(1 \le i \le 2(r+1))$.
Clearly, $H$ is isomorphic to the complete $(r+1)$-partite graph $K_{n, \ldots,n}$.
In \cite{PetArXiv}, it was proved that $H$ has an interval 
$\big( \frac{n(3r+1)}{2}-1\big)$-coloring $\gamma$ such that for any $1 \le i \le r+1$ and $1 \le j \le \frac{n}{2}$,
$$S\left(v_j^{(2i - 1)},\gamma \right) = S\left(v_j^{(2i)},\gamma\right) =\left[j + (i-1)\frac{n}{2}, j + (2r+i-1)\frac{n}{2}-1\right].$$

Define an edge-coloring $\varphi$ of $G$ as follows: for any $e \in E(G)$, let
\begin{equation*}
  \varphi(e)=\begin{cases}
    \gamma(e) + 1,  &\text{if } e \in E(H),\\
    j+(i-1)\frac{n}{2},       &\text{if } e=w_1v^{(2i-1)}_j, 1 \le i \le r+1,   1 \le j \le \frac{n}{2},\\
    j+(2r+i-1)\frac{n}{2} + 1,&\text{if } e=w_1v^{(2i)}_j,   1 \le i \le r-1,   1 \le j \le \frac{n}{2},\\
    j+(2r+i-1)\frac{n}{2} + 1,&\text{if } e=w_2v^{(2i-1)}_j, 1 \le i \le r+1,   1 \le j \le \frac{n}{2},\\
    j+(i-1)\frac{n}{2},       &\text{if } e=w_2v^{(2i)}_j,   1 \le i \le r-1,   1 \le j \le \frac{n}{2}.
  \end{cases}
\end{equation*}
Let us prove that $\varphi$ is a proper edge-coloring of $G$ with $\frac{n(3r+1)}{2}+1$ colors such that $\mathrm{def}(G, \varphi) = \mathrm{def}(w_1,\varphi)+\mathrm{def}(w_2,\varphi) = nr+2$. 

By the definition of $\varphi$, we have

\begin{description}
\item[1')] for $1 \le i \le r+1, 1 \le j \le \frac{n}{2},$ 
$$S\left(v_j^{(2i-1)},\varphi\right) = 
\left[j + (i - 1)\frac{n}{2}, j + (i + 2r- 1)\frac{n}{2}+1\right],$$
\item[2')] for $1 \le i \le r-1, 1 \le j \le \frac{n}{2},$
$$S\left(v_j^{(2i)},\varphi\right) = 
\left[j + (i - 1)\frac{n}{2}, j + (i + 2r- 1)\frac{n}{2}+1\right],$$
\item[3')] for $i \in \{ r, r+1 \}, 1 \le j \le \frac{n}{2},$
$$S\left(v_j^{(2i)},\varphi\right) = 
\left[j + (i - 1) \frac{n}{2}+1, j + (i + 2r- 1) \frac{n}{2}\right],$$
\item[4')] $S\left(w_1,\varphi\right) = \bigcup\limits_{i=1}^{r+1} \bigcup\limits_{j=1}^{\frac{n}{2}} \{ j + (i-1)\frac{n}{2} \} \cup \bigcup\limits_{i=1}^{r-1} \bigcup\limits_{j=1}^{\frac{n}{2}} \{ j + (2r+i-1)\frac{n}{2} +1\}=$
$\bigcup\limits_{i=1}^{r+1}\left[(i-1)\frac{n}{2}+1,i\frac{n}{2}\right]\cup
\bigcup\limits_{i=1}^{r-1}\left[(2r+i-1)\frac{n}{2}+2,(2r+i)\frac{n}{2}+1\right]=$
$\left[1, (r+1)\frac{n}{2}\right]\cup\left[2+ 2r\frac{n}{2},1+(3r-1)\frac{n}{2}\right],$
\item[5')] $S\left(w_2,\varphi\right) = \bigcup\limits_{i=1}^{r-1} \bigcup\limits_{j=1}^{\frac{n}{2}} \{ j + (i-1)\frac{n}{2} \} \cup 
\bigcup\limits_{i=1}^{r+1} \bigcup\limits_{j=1}^{\frac{n}{2}} \{ j + (2r+i-1)\frac{n}{2} +1\} =$
$\bigcup\limits_{i=1}^{r-1} \left[(i-1)\frac{n}{2}+1,i\frac{n}{2}\right] \cup\bigcup\limits_{i=1}^{r+1} \left[(2r+i-1)\frac{n}{2}+2,(2r+i)\frac{n}{2}+1\right]=$
$\left[1, \frac{n(r-1)}{2}\right]\cup\left[nr+2,\frac{n(3r+1)}{2}+1\right].$
\end{description}

This shows that for any $v \in V(G)\setminus \{w_1,w_2\}$, $S(v, \varphi)$ is an integer interval. Now we consider the spectrums of the vertices $w_1$ and $w_2$. By $4')$ and $5')$, we obtain

\begin{description}
\item[a')]$|S(w_1,\varphi)|= |S(w_2,\varphi)|= rn =d_G(w_1) =d_G(w_2),$
\item[b')] $\mathrm{def}(w_1,\varphi) = \overline{S}(w_1,\varphi) - \underline{S}(w_1,\varphi) - |S(w_1,\varphi)| + 1 = 
\frac{n(3r-1)}{2} - nr  + 1 =  \frac{n(r-1)}{2} + 1,$
\item[c')] $\mathrm{def}(w_2,\varphi) = \overline{S}(w_2,\varphi) - \underline{S}(w_2,\varphi) - |S(w_2,\varphi)| + 1 = 
\frac{n(3r+1)}{2} - nr  + 1 =  \frac{n(r+1)}{2} + 1.$
\end{description}
This implies that $\varphi$ is a proper edge-coloring of $G$ with $\frac{n(3r+1)}{2}+1$ colors such that $\mathrm{def}(G, \varphi) = \mathrm{def}(w_1,\varphi)+\mathrm{def}(w_2,\varphi) = nr+2$.
\end{proof}

Now we consider the deficiency of complete $(r+1)$-partite ($r\geq 3$) graphs with $r-1$ parts with the same size $n$ and two parts with the same size $n+1$. For such graphs we prove an upper bound on the deficiency.  

\begin{theorem}
For any $n,r \in \mathbb{N}$, if $r$ is odd, then for the complete $(r+1)$-partite graph $K_{n,\ldots,n,n+1, n+1}$,
$$\mathrm{def}(K_{n,\ldots,n,n+1,n+1}) \le n(r-1).$$
\end{theorem}

\begin{proof}
Let $G=K_{n,\ldots,n,n+1,n+1}$ be the complete $(r+1)$-partite graph and $V(G) = \left(\bigcup\limits_{i=1}^{r+1} V_i \right) \cup \{w_1, w_2\}$, where $V_1, V_2, \ldots, V_{r-1}$ and $V_{r} \cup \{w_1\}$, $V_{r+1} \cup \{w_2\}$ are parts of $G$ with sizes $|V_i|=n$ $(1 \le i \le r-1)$ and $|V_{r} \cup \{w_1\}| = |V_{r+1} \cup \{w_2\}| = n+1$. Also, let $H=G\left[ \bigcup\limits_{i=1}^{r+1}V_i \right]$ and $V_i=\left\{v_j^{(i)}: 1 \le j \le n\right\}$ $(1 \le i \le r+1)$.
Clearly, $H$ is isomorphic to the complete $(r+1)$-partite graph $K_{n, \ldots,n}$.
In \cite{PetArXiv}, it was proved that $H$ has an interval 
$\big(\frac{n(3r+1)}{2}-1\big)$-coloring $\alpha$ such that for any $1 \le i \le \frac{r+1}{2}$ and $1 \le j \le n$,
\begin{gather*}
S \left(v_j^{(2i - 1)},\alpha \right) = S \left(v_j^{(2i)},\alpha \right) = [j + (i-1)n, j + (i-1)n + rn -1] = \\
= [j + (i - 1)n, j + (i+r-1)n -1].
\end{gather*}
Define an edge-coloring $\beta$ of $G$ as follows: for any $e \in E(G)$, let
\begin{equation*}
  \beta(e)=\begin{cases}
    \alpha(e) + 1,      &\text{if } e \in E(H),\\
    j+(r+i-1)n + 1,     &\text{if } e=w_1v^{(2i-1)}_j, 1 \le i < \frac{r+1}{2}, 1 \le j \le n,\\
    j+(i-1)n,           &\text{if } e=w_1v^{(2i)}_j,   1 \le i \le \frac{r+1}{2},   1 \le j \le n,\\
    j+(i-1)n,           &\text{if } e=w_2v^{(2i-1)}_j, 1 \le i \le \frac{r+1}{2}, 1 \le j \le n,\\
    j+(r+i-1)n + 1,     &\text{if } e=w_2v^{(2i)}_j,   1 \le i < \frac{r+1}{2},   1 \le j \le n,\\
    nr + 1,             &\text{if } e=w_1w_2.
  \end{cases}
\end{equation*}
Let us prove that $\beta$ is a proper edge-coloring of $G$ with $\frac{n(3r+1)}{2}$ colors such that $\mathrm{def}(G, \beta) = \mathrm{def}(w_1,\beta)+\mathrm{def}(w_2,\beta) = n(r-1)$. 

By the definition of $\beta$, we have

\begin{description}
\item[1)] for $1 \le i < \frac{r+1}{2}, 1 \le j \le n,$ 
$$S \left(v_j^{(2i-1)},\beta \right) = S \left(v_j^{(2i)},\beta \right) = 
[j + (i - 1)n, j + (i + r- 1)n+1],$$
\item[2)] for $i = \frac{r+1}{2}, 1 \le j \le n,$
$$S \left(v_j^{(2i-1)},\beta \right) = S \left(v_j^{(2i)},\beta \right) = 
[j + (i - 1) n, j + (i + r- 1) n],$$
\item[3)] $S \left(w_1,\beta \right) = S \left(w_2,\beta \right) = 
\bigcup\limits_{i=1}^{\frac{r+1}{2}} \bigcup\limits_{j=1}^{n} \{ j + (i-1)n \} \cup 
\bigcup\limits_{i=1}^{\frac{r-1}{2}} \bigcup\limits_{j=1}^{n} \{ j + (r+i-1)n +1\} \cup \{nr+1\}= $\\
$=\bigcup\limits_{i=1}^{\frac{r+1}{2}} \left[(i-1)n+1,in\right]\cup
\bigcup\limits_{i=1}^{\frac{r-1}{2}} \left[(r+i-1)n+2,(r+i)n+1\right]\cup\{nr+1\}=$\\
$=\left[1, \frac{r+1}{2}n\right]\cup\left[nr+1,(r+\frac{r-1}{2})n+1\right].$
\end{description}

This shows that for any $v \in V(G)\setminus \{w_1, w_2\}$, $S(v, \beta)$ is an integer interval. Now we consider the spectrums of the vertices $w_1$ and $w_2$. By $3)$, we obtain
\begin{description}
\item[a)]$|S(w_1,\beta)| = |S(w_2,\beta)| = nr+1 = d_G(w_1) = d_G(w_2),$
\item[b)] $\mathrm{def}(w_1,\beta) = \overline{S}(w_1,\beta) - \underline{S}(w_1,\beta) - |S(w_1,\beta)| + 1 = 
(r+ \frac{r-1}{2})n - rn =  \frac{r-1}{2}n,$
\item[c)] $\mathrm{def}(w_2,\beta) = \overline{S}(w_2,\beta) - \underline{S}(w_2,\beta) - |S(w_2,\beta)| + 1 = 
(r+ \frac{r-1}{2})n - rn =  \frac{r-1}{2}n.$
\end{description}

This implies that $\beta$ is a proper edge-coloring of $G$ with $\frac{n(3r+1)}{2}$ colors such that $\mathrm{def}(G, \beta) = \mathrm{def}(w_1,\beta)+\mathrm{def}(w_2,\beta) = n(r-1)$.
\end{proof}

Finally we consider the deficiency of some class of complete $4$-partite graphs. Here we prove an upper bound on the deficiency.  

\begin{theorem}
For any $l,m,n \in \mathbb{N}$,
$$\mathrm{def}(K_{l,m,n,l+m+n}) \le \min\{l^2,m^2,n^2\}.$$
\end{theorem}
\begin{proof}
Without loss of generality we may assume that $l \le m \le n.$\\
Let $V(G) = W \cup U \cup V \cup T$, where $W = \{ w_i : 1 \le i \le l\}, U = \{ u_i : 1 \le i \le m\}$, $V = \{ v_i : 1 \le i \le n\}$ and $T= \{ t_i : 1 \le i \le l+m+n\}$. Also, let $H_1 = G[W \cup U \cup V]$ and $H_2=(G[W]\cup G[U]\cup G[V]) \lor G[T]$. Clearly, $H_1$ is isomorphic to the graph $K_{l,m,n}$ and $H_2$ is isomorphic to the graph $K_{l+m+n, l+m+n}$.\\

We define an edge-coloring $\alpha$ of $H_1$ as follows:
\begin{description}
\item[1)] for $1\le i\le m$, $1\le j\le n$, let $\alpha(u_iv_j) = l+i+j-1$;
\item[2)] for $1\le p\le l$, $1\le j\le n$, let $\alpha(w_pv_j) = p+j-1$;
\item[3)] for $1\le i\le m$, $1\le p\le l$, let $\alpha(u_iw_p) = l+n+i+p-1$.
\end{description}

By the definition of $\alpha$, we have
\begin{description}
\item[a)]for $1\le j\le n$, $S(v_j, \alpha) = 
\bigcup\limits_{i=1}^m \{l+i+j-1\} \cup \bigcup\limits_{p=1}^l \{p+j-1\}=$\\
$[j+l, j+l + m -1] \cup [j,j+l-1] = [j, j+ l + m - 1]$,
\item[b)] for $1\le i\le m $, $S(u_i,\alpha) = 
\bigcup\limits_{j=1}^ n \{l+i+j-1\} \cup \bigcup\limits_{p=1}^l \{l+n+i+p-1\}  = $\\
$[i+l, i+l+n-1] \cup [i+l+n, i+2l+n-1] = [i+l, i+2l+n-1]$,
\item[c)]for $1\le p\le l$, $S(w_p,\alpha) = 
\bigcup\limits_{j=1}^n \{p+j-1\} \cup \bigcup\limits_{i=1} ^ m \{l+n+i+p-1\}$=\\
$[p, p+n-1] \cup [p+l+n, p+l+n+m-1]$.
\end{description}
It is easy to see that $\alpha$ is a proper edge-coloring and for any $z \in U \cup V$, $\mathrm{def}(z, \alpha)=0$, and for any $w \in W$, $\mathrm{def}(w, \alpha)=l$. This implies that
$$\mathrm{def}(H_{1})\le \mathrm{def}(H_{1},\alpha)= \sum_{p=1}^{l} \mathrm{def}(w_p, \alpha) = l^{2}.$$

Let $L_1=USE(V, \alpha)$, $L_2=USE(U, \alpha)$, $L_3=USE(W, \alpha)$ and $Q_1=USE(W \cup U \cup V, \alpha)$. By the definition of $\alpha$, it follows that 
\begin{description}
\item[1')] $L_1 =(l+m,l+m+1,\ldots, l+m+n-1),$
\item[2')] $L_2 =(2l+n,2l+n+1,\ldots, 2l+n+m-1),$
\item[3')] $L_3 =(l+m+n,l+m+n+1,\ldots, 2l+m+n-1).$
\end{description}

From here, we obtain
\begin{description}
\item[a')] $\overline{L_1} \le \overline{L_3} = \overline{L_2}$
\item[b')] $\underline{L_2} \le \underline{L_3}$
\item[c')] $\underline{L_2}-1 \le \overline{L_1}$
\item[d')] $\underline{L_1} \le \underline{L_2}$
\end{description}

This shows that $Q_1=(q_1, q_2, \ldots, q_{l+m+n})$ is a continuous sequence, moreover $\underline{Q_1} = \underline{L_1}$ and $\overline{Q_1}=\overline{L_2}$. Let $Q_2=(q_1+1, q_2+1, \ldots, q_{l+m+n}+1)$. By Lemma \ref{Hayklemma}, we obtain that $H_2$ has a proper edge-coloring $\beta$ such that for any $z \in V(H_2)$, $S(z, \beta)$ is an integer interval and $LSE(T, \beta)=LSE(W \cup U \cup V, \beta)=Q_2$.\\
Now we are able to define an edge-coloring $\gamma$ of $G$.
For any $e \in E(G)$, let
\begin{equation*}
  \gamma(e)=\begin{cases}
    \alpha(e), & \text{if } e \in E(H_1),\\
    \beta(e), & \text{if } e \in E(H_2).
  \end{cases}
\end{equation*}
By the definition of $\gamma$, we have that for any $t \in T$, $S(t, \gamma) = S(t, \beta)$, hence $\mathrm{def}(t, \gamma)=0$. Next by the definitions of $Q_1, Q_2$ and $\alpha, \beta$, we obtain that for any $z \in W \cup U \cup V,$ $\overline{S}(z, \alpha)+1= \underline{S}(z, \beta)$, and hence $\mathrm{def}(z, \gamma)=\mathrm{def}(z, \alpha)$, since $\mathrm{def}(z, \beta)=0$ for any $z \in W \cup U \cup V$. This implies that
$$\mathrm{def}(G) \le \mathrm{def}(G,\gamma) = \sum_{v \in V(G)} \mathrm{def}(v, \gamma) = 
\sum_{v \in W \cup U \cup V} \mathrm{def}(v, \gamma) = \sum_{v \in W \cup U \cup V} \mathrm{def}(v, \alpha) \le \min\{l^2,m^2,n^2\}.$$

\end{proof}


\bigskip

\begin{thebibliography}{99}


\bibitem{AsrKam} A.S. Asratian, R.R. Kamalian, Interval colorings of edges of a
multigraph, Appl. Math. 5 (1987) 25-34 (in Russian).

\bibitem{AsrKamJCTB} A.S. Asratian, R.R. Kamalian, Investigation on interval
edge-colorings of graphs, J. Combin. Theory Ser. B 62 (1994) 34-43.

\bibitem{AsrDenHag} A.S. Asratian, T.M.J. Denley, R. Haggkvist, Bipartite Graphs and their Applications, Cambridge University Press, Cambridge, 1998.

\bibitem{B-OD-BHal} M. Borowiecka-Olszewska, E. Drgas-Burchardt, M. Ha\l uszczak, On the structure and deficiency of $k$-trees with bounded degree, Discrete Appl. Math. 201 (2016) 24-37.

\bibitem{BouchHertzDesau} M. Bouchard, A. Hertz, G. Desaulniers, Lower bounds and a tabu search algorithm for the minimum deficiency problem, J. Comb. Optim. 17 (2009) 168-191.

\bibitem{CasselgrenKhachatrianPetrosyan} C.J. Casselgren, H.H. Khachatrian, P.A. Petrosyan, Some bounds on the number of colors in interval and cyclic interval
edge colorings of graphs, Discrete Math. 341 (2018) 627-637.

\bibitem{GiaroKubaleMalaf1} K. Giaro, M. Kubale, M. Ma\l afiejski, On the deficiency of bipartite graphs, Discrete Appl. Math. 94 (1999) 193-203.

\bibitem{GiaroKubaleMalaf2} K. Giaro, M. Kubale, M. Ma\l afiejski, Consecutive colorings of the edges of general graphs, Discrete Math. 236 (2001) 131-143.

\bibitem{Kampreprint} R.R. Kamalian, Interval colorings of complete bipartite graphsand trees, preprint, Comp. Cen. of Acad. Sci. of Armenian SSR, Yerevan, 1989 (in Russian).

\bibitem{KamDiss} R.R. Kamalian, Interval edge colorings of graphs, Doctoral Thesis, Novosibirsk, 1990.

\bibitem{KhachOuterplanar} H.H. Khachatrian, Deficiency of outerplanar graphs, Proceedings of the Yerevan State University, Physical and Mathematical Sciences, 2017, 51 (1), pp. 3-9.

\bibitem{Kubale} M. Kubale, Graph Colorings, American Mathematical Society, 2004.

\bibitem{PetDM} P.A. Petrosyan, Interval edge-colorings of complete graphs and $n$-dimensional cubes, Discrete Math. 310 (2010) 1580-1587.

\bibitem{PetArXiv} P.A. Petrosyan, Interval colorings of complete balanced multipartite graphs, ArXiv:1211.5311 (2012).

\bibitem{PetKhachTan} P.A. Petrosyan, H.H. Khachatrian, H.G. Tananyan, Interval edge-colorings of Cartesian products of graphs I, Discuss. Math. Graph Theory 33(3) (2013) 613-632.

\bibitem{PetrosHrant} P.A. Petrosyan, H.H. Khachatrian, Further results on the deficiency of graphs, Discrete Appl. Math. 226 (2017) 117-126.

\bibitem{Seva} S.V. Sevast'janov, Interval colorability of the edges of a bipartite graph, Metody Diskret. Analiza 50 (1990) 61-72 (in Russian).

\bibitem{Schwartz} A. Schwartz, The deficiency of a regular graph, Discrete Math. 306 (2006) 1947-1954.

\bibitem{TePet} H.H. Tepanyan, P.A. Petrosyan, Interval edge-colorings of composition of graphs, Discrete Appl. Math. 217 (2017) 368-374.

\bibitem{West} D.B. West, Introduction to Graph Theory, Prentice-Hall, New Jersey, 2001.

\end{thebibliography}
\end{document}